\pdfoutput=1
    \documentclass{amsart}
    \usepackage{keytheorems}

\usepackage[mathscr]{eucal}
\usepackage{amssymb,amsmath}
\usepackage[utf8]{inputenc}
\usepackage[T1]{fontenc}
\usepackage{lmodern}
\usepackage{amsthm}
\usepackage{comment}
\usepackage{graphicx} 
\usepackage{zref-clever}
\zcsetup{cap,nameinlink}

\usepackage{todonotes}

\usepackage{mathtools}

\providecommand\given{} 
\DeclarePairedDelimiterX\Set[1]\{\}{\renewcommand\given{\SetSymbol[\delimsize]}#1}

\usepackage{tikz-cd}
\usepackage{lmodern}
\usepackage[dvipsnames]{xcolor}

\usepackage{microtype}
\newcommand{\cat}[1]{\mathscr{#1}}
\usetikzlibrary{arrows}
\usepackage[unicode,colorlinks]{hyperref}
\hypersetup{
  linkcolor=BrickRed,
  citecolor=Green,
  filecolor=Mulberry,
  urlcolor=NavyBlue,
  menucolor=BrickRed,
  runcolor=Mulberry,
  linkbordercolor=BrickRed,
  citebordercolor=Green,
  filebordercolor=Mulberry,
  urlbordercolor=NavyBlue,
  menubordercolor=BrickRed,
  runbordercolor=Mulberry
}
\usepackage{etoolbox,mathtools} 
\usepackage{xparse}
\usepackage[all]{xy}
\DeclareMathOperator{\pt}{pt}
\providecommand\given{}
\newcommand\LocSymbol[1][]{\nonscript\:#1\vert\allowbreak\nonscript\:\mathopen{}}
\DeclarePairedDelimiterX\LocArg[1]{(}{)}{\renewcommand\given{\LocSymbol[\delimsize]}#1}
\newcommand{\Loc}[1]{\operatorname{Loc}\LocArg{#1}}

\newcommand{\nc}{\newcommand}
\numberwithin{equation}{section}
\DeclareMathOperator{\im}{im}
\newkeytheorem{Thm}[name=Theorem, sibling=equation, Refname={Theorem,Theorems}]
\newkeytheorem{Lem}[name=Lemma, sibling=Thm, Refname={Lemma,Lemmas}]
\newkeytheorem{Prop}[name=Proposition, sibling=Thm, Refname={Proposition,Propositions}]
\newkeytheorem{Cor}[name=Corollary, sibling=Thm, Refname={Corollary,Corollaries}]
\newkeytheorem{Exa}[name=Example, sibling=Thm, Refname={Example,Examples},style=remark]
\newkeytheorem{Def}[name=Definition, sibling=Thm, Refname={Definition},style=remark]
\newkeytheorem{Rem}[name=Remark, sibling=Thm, Refname={Remark,Remarks},style=remark]
\newkeytheorem{Cons}[name=Construction, sibling=Thm, Refname={Construction},style=remark]
\newkeytheorem{Not}[name=Notation, sibling=Thm, Refname={Notation},style=remark]
\newkeytheorem{Hyp}[name=Hypothesis, sibling=Thm, Refname={Hypothesis},style=remark]
\nc{\dmo}{\DeclareMathOperator}
\dmo{\surj}{surj}
\nc{\bbZ}{\mathbb{Z}}
\nc{\bbN}{\mathbb{N}}

\usepackage[
    backend=biber,
    style=alphabetic,
    maxalphanames=4,
    minalphanames=4,
    maxbibnames=5,
    maxcitenames=5,
    backref=true,
    sortlocale=de_DE,
    natbib=true,
    doi=false,
    isbn=false,
    url=false
]{biblatex}
\DeclareMathOperator{\Spec}{Spec}

\DeclareMathOperator{\Hom}{Hom}

\makeatletter

\newcommand{\bbF}{\mathbb{F}}

\DeclareMathOperator{\Sm}{Sm}
\DeclareMathOperator{\RadId}{RadIdl}
\DeclareMathOperator{\Der}{D}
\DeclareMathOperator{\Idl}{Idl}
\makeatother

\usepackage{dsfont}
\newcommand{\unit}{\mathds{1}}

\usepackage{bbm}
\DeclareMathOperator{\Mod}{Mod}

\newcommand{\ba}[1]{\mathbb{#1}}
\newcommand{\meet}{\land}
\newcommand{\join}{\lor}
\DeclareMathOperator{\At}{At}
\DeclareMathOperator{\Clop}{Clop}
\newcommand{\R}[1]{\mathfrak{R}(#1)}
\newcommand{\Idem}[1]{\mathrm{IdemIdl}(#1)}
\DeclareMathOperator{\Ult}{Ult}
\newcommand{\ideal}[1]{\mathfrak{J}_{#1}}

\newcommand{\BRings}{\mathbf{BoolRings}}
\newcommand{\BAlg}{\mathbf{BoolAlg}}
\DeclareMathOperator{\Spc}{Spc}

\DeclareMathOperator{\colim}{colim}
\zcRefTypeSetup{equation}{
  Name-sg-ab =,
  name-sg-ab =,
  Name-pl-ab =,
  name-pl-ab =,
  abbrev,
  refbounds = {,(,),},
}
\newkeytheoremstyle{mainthmstyle}{
    headformat=\NAME\ \NUMBER\NOTE
}
\newkeytheorem{MainThm}[
    name=Theorem,
    style=mainthmstyle,
    counter-format=\Alph*,
    Refname={Theorem,Theorems}
]
    \title{A non-spatial frame of smashing ideals}
    \author[T.~Barthel]{Tobias Barthel}
    \author[D.~Heard]{Drew Heard}
    \date{\today}
    \makeatletter
\patchcmd{\@setaddresses}{\indent}{\noindent}{}{}
\patchcmd{\@setaddresses}{\indent}{\noindent}{}{}
\patchcmd{\@setaddresses}{\indent}{\noindent}{}{}
\patchcmd{\@setaddresses}{\indent}{\noindent}{}{}
\makeatother

\address{Tobias Barthel, Max Planck Institute for Mathematics, Vivatsgasse 7, 53111 Bonn, Germany}
\email{tbarthel@mpim-bonn.mpg.de}
\urladdr{\href{https://sites.google.com/view/tobiasbarthel/home}{https://sites.google.com/view/tobiasbarthel/home}}
\address{Drew Heard, Department of Mathematical Sciences, Norwegian University of Science and Technology, Trondheim, Norway}
\email{drew.k.heard@ntnu.no}
\urladdr{\href{https://folk.ntnu.no/drewkh/}{https://folk.ntnu.no/drewkh}}
    
\begin{document}
    \begin{abstract}
        We construct commutative rings whose derived categories have non-spatial frames of smashing ideals, thereby answering in the negative questions of Balchin--Stevenson and  Aoki.
    \end{abstract}
    \maketitle
    
    \vspace{-1em}
    \tableofcontents

    \vspace{-1em}
    \section{Introduction}

    Tensor triangular geometry offers a geometric perspective on essentially small tensor triangulated (`tt') categories $\cat K$, by deconstructing $\cat K$ as a sheaf of tt-categories on its Balmer spectrum $\Spc(\cat K)$. The topological space $\Spc(\cat K)$ carries the same information as the frame $\mathrm{Thick}_{\otimes}(\cat K)$ of radical thick tensor-ideals in $\cat K$; via Stone (plus Hochster) duality, this is equivalent to the statement that $\mathrm{Thick}_{\otimes}(\cat K)$ is spatial. 

    Various extensions of the spectrum to the `big' setting of rigidly-compactly generated tt-categories $\cat T$ have been proposed over the years, notably based on the poset of smashing ideals $\Sm(\cat T)$. The theory of finite localizations provides an embedding $\mathrm{Thick}_{\otimes}(\cat T^{\omega}) \hookrightarrow \Sm(\cat T)$ of posets, whose deviation from being surjective measures the failure of the generalized telescope conjecture for $\cat T$. Balmer, Krause, and Stevenson \cite{BalmerKrauseStevenson2020Frame} proved that $\Sm(\cat T)$ forms a frame, leading to the \emph{smashing spectrum} as the space representing $\Sm(\cat T)$ via Stone duality in \cite{BalchinStevenson21pp}. As before, this provides a faithful geometric representation only when $\Sm(\cat T)$ is spatial.
    
    In \cite[page 2]{BalchinStevenson21pp}, Balchin and Stevenson raised the question if $\Sm(\cat T)$ is always spatial. Some earlier attempts to prove spatiality were subsequently found to contain gaps; see \zcref{rem:history} for further details.  More specifically, Aoki \cite[Question 1.6]{aoki2024smashing} asked whether $\Sm(\Der(A))$ is spatial for any commutative ring $A$. The goal of this paper is to answer both questions in the negative.    
    
    \begin{MainThm}\label{mainthm:main}
        The frame of smashing ideals in $\Der(A)$ is in general not spatial. 
    \end{MainThm}

    The proof of our main theorem consists of two steps: First, reducing the problem to a purely algebraic question about idempotent ideals; and second, constructing a suitable family of commutative rings that satisfy the required properties. 

    The reduction step begins with the observation that the poset of idempotent ideals in a commutative ring $A$ forms a frame $\Idem{A}$, which we wish to relate to the frame of smashing ideals in the derived category $\Der(A)$. Indeed, to any ideal $I$ in $A$, we can assign a full subcategory of $\Der(A)$ on objects $M$ satisfying $I \cdot H_*(M)  = 0$. If $I$ is idempotent, it turns out that this subcategory is the image of a smashing localization functor on $\Der(A)$, whose kernel we denote by $\Phi_A(I)$. This enables us to show (\zcref{thm:frame-embedding}):
    \begin{MainThm}\label{mainthm:smashing}
        For every commutative ring $A$, the assignment
        \[
            \Phi_A\colon \Idem{A}\longrightarrow\Sm(\Der(A)), \quad I \mapsto {}^{\perp}\{M\mid I\cdot H_*(M)=0\}
        \]
        is an injective frame homomorphism.
    \end{MainThm}
    
    In order to prove \zcref{mainthm:main}, it therefore suffices to construct a commutative ring whose frame of idempotent ideals is not spatial. To this end, we introduce a family of rings $\R{\ba{B}}$ called \emph{nil-Cantor rings}, taking as input Boolean algebras~$\ba{B}$ and an implicit field $k$. For $\ba{B} = \ba{F}_{[n]}$ a free Boolean algebra on $n$ generators, the nil-Cantor ring is given by the truncated polynomial ring
        \[
            \R{\ba{F}_{[n]}} \cong k[x_1,x_2,\ldots,x_{2^n}]/(x_i^2 \mid 1 \leq i \leq 2^n).
        \]
    Following a coordinate-free approach, we extend the construction of $\R{\ba{B}}$ to a continuous functor on Boolean algebras and injective homomorphisms, by tracking the splitting of atoms under homomorphisms between finite Boolean algebras (\zcref{ssec:construction}). This determines the value $\R{\ba{B}}$ on any Boolean algebra $\ba{B}$ as a suitable directed limit of truncated polynomial rings. 
    
    Of particular interest to us is the case when $\ba{B}$ is atomless, i.e., when $\ba{B}$ arises as the Boolean algebra of clopen subsets of a Stone space without isolated points such as Cantor space, thereby justifying the nomenclature. Our main result about the rings $\R{\ba{B}}$ is the following (\zcref{prop:idempoints}), which finishes the proof of \zcref{mainthm:main}:

    \begin{MainThm}\label{mainthm:idempotent}
        For any atomless Boolean algebra $\ba{B}$, the frame of idempotent ideals in $\R{\ba{B}}$ has cardinality $\geq \lvert\ba{B}\rvert$, but only a single point:
            \[
                \pt(\Idem{\R{\ba{B}}}) = \{\ast\}.
            \]
        In particular, if $\ba{B}$ is nontrivial, then $\Idem{\R{\ba{B}}}$ is not spatial.
    \end{MainThm}
    
   Since the spectrum of $\Der(\R{\ba{B}})^{\omega}$ is a singleton, it follows that nil-Cantor rings provide examples which witness an arbitrarily bad failure of the telescope conjecture.
    
    \subsection{Acknowledgements}
    
    Both authors are grateful to the Max Planck Institute for Mathematics, where the results of this paper were found during a visit of the second-named author to the first-named one. We would like to thank Paul Balmer, Frank Gounelas, Henning Krause, and Greg Stevenson for useful comments. TB is supported by the ERC under Horizon Europe~(grant~No.~101042990).
    
    \subsubsection*{Use of generative AI}
    The initial proof strategy emerged from an extended interaction with OpenAI's ChatGPT. In particular, ChatGPT first observed that the results of Hebestreit--Scholze \cite{hebestreit2024note} could be used to construct the frame embedding $\Phi_A$ and identified their binary-tree ring as an example for which $\Idem{A}$ is not spatial. The formulation and generalization in terms of atomless Boolean algebras as well as the present argument were developed by the authors, with further expositional input from ChatGPT. The authors have written the entire paper and take full responsibility for its content.

    \section{The frame of smashing tensor ideals}
    \subsection{Frame theory}
    We begin with some generalities on frame theory.
    \begin{Def}
        A \emph{frame} is a complete lattice $L$, satisfying the distributivity law
        \[
        a \wedge \bigvee_i b_i = \bigvee_i (a \wedge b_i).
        \]
    A \emph{frame homomorphism} is a homomorphism of posets that preserves finite meets (including the top element) and arbitrary joins (including the bottom element). 
    \end{Def}
    \begin{Def}
        A \emph{point} of $L$ is a morphism of frames $p \colon L \to \{ 0 < 1 \}$. A frame is \emph{spatial} if it has enough points, i.e., whenever $a \nleq b$ there exists a point $p$ with $p(a) = 1$ and $p(b) = 0$. 
    \end{Def}
    \begin{Rem}
        The condition of being spatial admits several equivalent formulations. For example, it means that $L \simeq \mathcal{O}(X)$, i.e., $L$ is isomorphic to the frame of open subsets of some topological space $X$. More specifically, we can take $X = \pt(L)$ to be the set of points of $L$, with the topology whose open subsets are 
        \[
      U_a=\{p\in\pt(L)\mid p(a)=1\}
    \]
    for $a \in L$. 
    \end{Rem}
    \begin{Def}
        A proper element $q<1$ of a frame $L$ is called \emph{prime} if
        \[
            a \wedge b\leq q \implies a \leq q \text{ or } b \leq q.
        \]
    \end{Def}
    \begin{Lem}\label{lem:points-prime-elements}
        Points of a frame $L$ are in bijection with its prime elements.
    \end{Lem}
    \begin{proof}
        This is a well-known result. The bijection is given by sending a prime element $q$ to the point $p_q$ defined by
        \[
        p_q(a) = \begin{cases} 
        0 & \text{ if } a \le q \\
        1 & \text{ if } a \not \le q
        \end{cases}
        \]
        and by sending a point $p$ to the element $q_p \coloneqq \bigvee \Set{a \in L \given p(a) = 0}$. It is straightforward to see that these are inverse bijections. 
        \end{proof}
    \begin{Lem}\label{lem:subframe-of-spatial-frame}
        Let $f \colon L \to K$ be an injective frame homomorphism. If $K$ is spatial, then so is $L$. 
    \end{Lem}
    \begin{proof}
        Note first that $f$ reflects order. Indeed, if $f(a) \le f(b)$, then $f(a \vee b) = f(a) \vee f(b) = f(b)$, and so $a \vee b = b$ since $f$ is injective, i.e., $a \le b$. Therefore if $a \nleq b$ in $L$, then $f(a) \nleq f(b)$ in $K$ as well. Since $K$ is spatial, there exists a point $p \colon K \to \{ 0 < 1 \}$ with $p(f(a)) = 1$ and $p(f(b)) = 0$. The composite $p \circ f$ is then a point of $L$ satisfying $(p \circ f)(a) = 1$ and $(p \circ f)(b) = 0$, as required. 
    \end{proof}
    \subsection{Smashing ideals}
    We now let $\cat T$ be a rigidly-compactly generated tensor-triangulated category. We recall that this means that the subcategory $\cat T^c \subseteq \cat T$ of compact objects coincides with that of rigid (i.e., dualizable)
    objects, and that $\cat T^c$ generates $\cat T$ as a localizing subcategory. 
    \begin{Def}
        A localizing tensor ideal $\cat S \subseteq \cat T$ is called \emph{smashing} if it is the kernel of a Bousfield localization functor $L \colon \cat T \to \cat T$ which preserves coproducts. We write $\Sm(\cat T)$ for the poset\footnote{That this forms a set is a consequence of \cite[Theorem 1.3 and Corollary 2.15]{BalmerKrauseStevenson2020Frame}.} of smashing ideals, ordered by inclusion.
    \end{Def}
    \begin{Rem}\label{rem:smashing-ideals-and-idempotents}
        Smashing ideals can equivalently be described by tensor-idempotent triangles. Indeed, if $\cat S$ is a smashing ideal, then evaluating its localization triangle at the tensor unit gives an idempotent triangle
        \[
            e_{\cat S}\longrightarrow\unit \longrightarrow f_{\cat S}\longrightarrow\Sigma e_{\cat S}
        \]
        such that $e_{\cat S}\otimes e_{\cat S}\simeq e_{\cat S}$ and $f_{\cat S}\otimes f_{\cat S}\simeq f_{\cat S}$. Moreover, the associated localization and colocalization functors are given by
        \[
            L_{\cat S}(-)\simeq f_{\cat S}\otimes-,\qquad \Gamma_{\cat S}(-)\simeq e_{\cat S}\otimes-.
        \]
        Conversely, every idempotent triangle determines a smashing ideal $\ker(f\otimes-)=\im(e\otimes-)$. This gives an inclusion-preserving bijection between smashing ideals and idempotent triangles \cite[Theorems 2.13 and 3.5]{BalmerFavi2011Generalized}.
    \end{Rem}
    The following was proved in \cite[Theorem 5.5]{BalmerKrauseStevenson2020Frame}.
    \begin{Thm}[Balmer--Krause--Stevenson]
    The poset $\Sm(\cat T)$ of smashing ideals, ordered by inclusion, forms a frame with bottom element the zero ideal, and top element $\cat T$. The join of a family $\Set{\cat S_\lambda  \given \lambda \in \Lambda}$ of smashing ideals is $\Loc{\cat S_\lambda \given \lambda \in \Lambda}$, and finite meets are given by intersection. 
    \end{Thm}
 \begin{Rem}
    In contrast to finite meets, infinite meets in $\Sm(\cat T)$ are not simply given by intersections, since the intersection need not be smashing; see \cite[Remark~5.12]{BalmerKrauseStevenson2020Frame} for an explicit example. The meet is the largest smashing ideal contained in the intersection.
\end{Rem}  
    \begin{Rem}\label{rem:history}
        The appendix to \cite{BalchinStevenson21pp} discusses some of the previous attempts and their failure to answer the question of when the frame $\Sm(\cat T)$ is spatial; we include a brief summary here. In her thesis, Wagstaffe \cite[Theorem 6.2.2]{Wagstaffe2021Definability} gave an argument that the frame $\Sm(\cat T)$ is spatial using techniques from model theory. Unfortunately, this used a lemma of Krause \cite[Lemma 4.11]{Krause2000Smashing} which used an incorrect definition of exact ideal, see \cite[p.~1224]{Krause2005Cohomological}. In fact, \cite[Example A.2.3]{BalchinStevenson21pp} explains why this proof cannot work. An earlier version of \cite{BalchinStevenson21pp} also contained a false proof that $\Sm(\cat T)$ was spatial, and \cite[Example A.2.2]{BalchinStevenson21pp} showed that their argument could not be repaired. 
        
        More recent work \cite{GomezEtAl2026Geometric} introduced the notion of geometric purity in order to study the question of spatiality of the frame of smashing ideals,  however the general question of spatiality has remained open in the rigidly-compactly generated setting. Finally, outside the rigidly-compactly generated setting, Aoki has constructed an example of a category which is not compactly generated whose frame of smashing ideals is not spatial \cite{aoki2024smashing}. 
    \end{Rem}
    \section{Idempotent ideals and smashing localizations}
    \subsection{The frame of idempotent ideals}
    We begin with the frame of idempotent ideals of an ordinary commutative ring $A$.
    \begin{Def}
        Let $\Idem{A}$ denote the poset of idempotent ideals in $A$, i.e., 
        \[
        \Idem{A} \coloneqq \Set{I \triangleleft A \given I^2 = I}. 
        \]
    \end{Def}
    \begin{Lem}\label{lem:idempotent-frame} The poset $\Idem{A}$ is a frame.  Its arbitrary joins and finite meets are given by
    \[ 
    \bigvee_{\alpha} I_\alpha=\sum_\alpha I_\alpha,  \qquad  I\wedge J=IJ. 
    \]
    \end{Lem}
    \begin{proof}
    Let $I=\sum_\alpha I_\alpha$. Since $I_\alpha^2=I_\alpha$, we have $I_\alpha\subseteq I^2$ for every $\alpha$, and hence $I\subseteq I^2$. The reverse inclusion is automatic, so $I$ is idempotent.  Similarly $(IJ)^2 = I^2J^2 = IJ$. Moreover, $IJ\subseteq I\cap J$. If $K\in \Idem{A}$ satisfies $K\subseteq I$ and $K\subseteq J$, then $K=K^2\subseteq IJ$. Thus $IJ$ is the meet of $I$ and $J$ in $\Idem{A}$. Finally,
    \[
    I\left(\sum_\alpha J_\alpha\right)=\sum_\alpha IJ_\alpha,
    \]
    which proves the frame distributivity law.
    \end{proof}
    \begin{Rem}
        Note that the ordinary poset $\Idl(A)$ of ideals in $A$ forms a complete lattice (although not in general a frame), where the meet is given by the intersection of ideals $I \cap J$, not the product $IJ$. In general, $I \cap J$ need not be idempotent, and $\Idem{A} \subseteq \Idl(A)$ is not a sublattice. 
    \end{Rem}

    We now wish to show that $\Idem{A}$ is always a subframe of $\Sm(\Der(A))$. This is perhaps a little surprising, since it is not initially clear how to construct a map between the two. Our approach relies on an $\infty$-categorical reflection principle which does not seem to be available in the purely triangulated setting. 

        \begin{Prop}\label{prop:left-adjoint}
            Let $\cat C$ be a rigidly-compactly generated\footnote{This proposition extends to any presentably symmetric monoidal stable $\infty$-category generated under colimits by a set of dualizable objects, but we will not require this generality here.} symmetric monoidal stable $\infty$-category, and let $\cat J \subseteq\cat C$ be a localizing tensor ideal which is closed under products. 
            Then the inclusion $\iota \colon \cat J\hookrightarrow\cat C$ admits a left adjoint $L$, and the endofunctor $\iota \circ L \colon \cat C \to \cat C$ is a smashing localization.
        \end{Prop}
        \begin{proof}
            Since $\cat J$ is stable and closed under products, it is automatically closed under all limits (\cite[Proposition~1.4.4.1]{HALurie}), and by definition it is closed under colimits. 

            Fix $d\in\cat J$ and let
            \[
                \cat J_d \coloneqq \Set{c\in\cat C\given\underline{\Hom}_{\cat C}(c,d)\in\cat J}.
            \]
    The subcategory $\cat J_d$ is localizing, because
    $\underline{\Hom}_{\cat C}(-,d)$ is exact and sends coproducts to products. If $g$ is dualizable, then
    \[
        \underline{\Hom}_{\cat C}(g,d)\simeq g^\vee\otimes d\in\cat J,
    \]
    since $\cat J$ is a tensor ideal. Thus $\cat J_d$ contains a set of
    generators of $\cat C$, and hence $\cat J_d=\cat C$.  Therefore $\cat J$ satisfies condition~(3) of \cite[Theorem~7.7]{RagimovSchlank2026Categorical}, so its inclusion admits a smashing left adjoint.
        \end{proof}
        \begin{Not}
            For $A$ a commutative ring, we let $\Der_{\infty}(A)$ denote the symmetric monoidal $\infty$-category whose homotopy category is $\Der(A)$. 
        \end{Not}
        \begin{Def}
            For $I \in \Idem{A}$, set
            \[
            \cat L^{\infty}_I \coloneqq \Set{M \in \Der_{\infty}(A) \given I \cdot H_n(M) = 0 \text{ for every } n \in \bbZ}
            \]
            and set $\cat L_I \coloneqq \operatorname{Ho}(\cat L^{\infty}_I)$. Thus,
            \[
            \cat L_I = \Set{M \in \Der(A) \given I \cdot H_n(M) = 0 \text{ for every } n \in \bbZ}.
            \]
        \end{Def}
\begin{Prop}
    The inclusion $\iota\colon\cat L_I\hookrightarrow\Der(A)$ admits a left adjoint $L_I\colon\Der(A)\to\cat L_I$, and  the endofunctor $\iota \circ L_I \colon \Der(A) \to  \Der(A)$ is a smashing localization.
\end{Prop}
\begin{proof}
   This will follow from the result at the level of $\infty$-categories by passing to homotopy categories. We therefore verify the conditions of \zcref{prop:left-adjoint} for the inclusion $\cat L_I^{\infty} \subseteq \Der_{\infty}(A)$. We first observe that because $\Der_{\infty}(A)$ is rigidly-compactly generated by its tensor unit $A$, all localizing subcategories are localizing ideals.  Since homology commutes with both products and coproducts, we deduce that $\cat L^{\infty}_I$ is closed under these operations. Similarly, closure under shifts is clear. Suppose that 
            \[
            M' \to M \to M''
            \]
            is a fiber sequence in $\Der_{\infty}(A)$ with $M',M'' \in \cat L^{\infty}_I$. The long exact sequence in homology shows that $I^2 \cdot H_n(M) = 0$ for every $n$. Since $I^2 = I$, it follows that $M \in \cat L^{\infty}_I$, as required. 
\end{proof}
\begin{Def}
    For $I\in\Idem{A}$, let
    \[
        \Phi_A(I)\coloneqq\ker(L_I) \in\Sm(\Der(A)).
    \]
    This defines a map
    \[
        \Phi_A\colon\Idem{A}\longrightarrow\Sm(\Der(A)).
    \]
\end{Def}
\begin{Rem}\label{rem:strictly-localizing}
    The subcategories $\Phi_A(I)$ and $\cat L_I$ determine each other. Indeed, $\Phi_A(I)$ is the kernel and $\cat L_I$ the essential image of the
    corresponding Bousfield localization. Hence passing to right respectively left orthogonals, we have
    \[
        \cat L_I = \Phi_A(I)^\perp \qquad \text{and} \qquad \Phi_A(I) = {}^{\perp} \cat L_I.
    \]
    The second equality uses the standard identity $\cat S={}^{\perp}(\cat S^\perp)$ for a strictly localizing subcategory $\cat S$; see \cite[Proposition~4.9.1]{Krause10}.
\end{Rem}
We will need the following auxiliary result. 
\begin{Lem}\label{lem:heart-subcategory}
       Let $\cat L \subseteq \Der(A)$ be a the full subcategory of local objects for some smashing localization on $\Der(A)$. 
       If $M \in \Der(A)$ satisfies $H_n(M) \in \cat L$ for every $n\in\bbZ$, then $M \in \cat L$.
\end{Lem}
\begin{proof}
    Note that $\cat L$ is closed under coproducts and products in $\Der(A)$. A sequential homotopy limit is the fibre of a map between countable products, and a sequential homotopy colimit the cofibre of a map between countable coproducts, so $\cat L$ is also closed under the formation of these. The standard $t$-structure on $\Der(A)$ is left and right complete \cite[Proposition~7.1.1.13]{HALurie}. Hence for every $M\in\Der(A)$,
    \[
        M \simeq \operatorname*{hocolim}_{a\to\infty} \tau_{\geq-a}M
    \]
        and
    \[
        \tau_{\geq-a} M \simeq \operatorname*{holim}_{b\to\infty} \tau_{\leq b} \tau_{\geq-a} M .
    \]
    Each $\tau_{\leq b}\tau_{\geq-a}M$ is bounded, hence obtained from the shifts $\Sigma^kH_k(M)$, $-a\leq k\leq b$, by finitely many extensions. By hypothesis, $\Sigma^kH_k(M)$ belongs to $\cat L$  for all $k$, therefore $\tau_{\leq b}\tau_{\geq-a}M \in \cat L$ for all $-a \leq b$, and the result follows.
\end{proof}
    \begin{Thm}\label{thm:frame-embedding}
        For every commutative ring $A$, the assignment
        \[
            \Phi_A\colon \Idem{A}\longrightarrow\Sm(\Der(A))
        \]
        is an injective frame homomorphism.
    \end{Thm}
    \begin{proof}
         We first show that $\Phi_A$ preserves arbitrary joins, i.e., that $\Phi_A(\sum_\alpha I_\alpha) = \Loc{(\Phi_A(I_{\alpha}))_{\alpha}}$. Let $J \coloneqq \sum_{\alpha} I_{\alpha}$ and set $\cat S \coloneqq \Loc{(\Phi_A(I_{\alpha}))_{\alpha}}$. The subcategory of $\cat S$-local objects is
    \[
    \begin{split}
    \cat S^\perp &=\bigcap_\alpha\Phi_A(I_\alpha)^\perp\\
    &= \bigcap_\alpha  \Set{M \in \Der(A) \given I_\alpha\cdot H_n(M)=0 \text{ for every }n\in\bbZ}\\
    &= \Set{M\in\Der(A) \given I_\alpha\cdot H_n(M)=0 \text{ for every }\alpha\text{ and every }n\in\bbZ}\\
    &= \Set{M\in \Der(A) \given  J\cdot H_n(M)=0\text{ for every }n\in\bbZ}\\
    &= \Phi_A(J)^\perp,
    \end{split}
    \]
    so that, by \zcref{rem:strictly-localizing}, we have
    \[
        \cat S={}^{\perp}(\cat S^\perp)
        ={}^{\perp}(\Phi_A(J)^\perp)
        =\Phi_A(J).
    \]
    We now show that $\Phi_A$ preserves finite meets, i.e., that $\Phi_A(I \meet J) = \Phi_A(I \cdot J) = \Phi_A(I) \cap \Phi_A(J)$. By definition, $\Phi_A(I \cdot J) = {}^{\perp}\cat L_{I \cdot J}$. It is clear that $\cat L_I,\cat L_J \subseteq \cat L_{I \cdot J}$, and so $\cat L_I \cup \cat L_J \subseteq \cat L_{I \cdot J}$. Taking orthogonals, we deduce that
    \begin{equation}\label{eq:inclusion-1}
    \Phi_A(I \cdot J) = {}^{\perp} \cat L_{I \cdot J} \subseteq {}^{\perp}(\cat L_I \cup \cat L_J) = {}^{\perp}\cat L_I \cap {}^{\perp}\cat L_J = \Phi_A(I) \cap \Phi_A(J).
    \end{equation}
    Conversely, suppose that $N \in \Mod(A)$ is a \emph{discrete} $A$-module with $IJ \cdot N = 0$, and consider the short exact sequence 
    \[
    0 \to JN \to N \to N/JN \to 0.
    \]
    This satisfies $I(JN) = 0$ and $J(N/JN) = 0$, so that $JN \in \cat L_I$ and $N/JN \in \cat L_J$.  Let $\cat P \coloneqq (\Phi_A(I) \cap \Phi_A(J))^{\perp}$, and observe that $\Phi_A(I),\Phi_A(J) \supseteq \Phi_A(I) \cap \Phi_A(J)$ implies $\cat L_I,\cat L_J \subseteq \cat P$, so that the above short exact sequence shows that $N \in \cat P$. 

    Now consider $M \in \cat L_{IJ}$. Recall that $\Phi_A(I) \cap \Phi_A(J) = \Phi_A(I) \meet \Phi_A(J)$ in $\Sm(\Der(A))$, and hence remains smashing, and so the right orthogonal $\cat P$ is closed under products and coproducts. Therefore, the previous paragraph and \zcref{lem:heart-subcategory} show that $M \in \cat P$, i.e., 
    \[
    \cat L_{IJ} \subseteq (\Phi_A(I) \cap \Phi_A(J))^{\perp}.
    \]
    Taking left orthogonals and using \zcref{rem:strictly-localizing} we deduce that 
  \begin{equation}\label{eq:inclusion-2}
    \Phi_A(I \cdot J) \supseteq \Phi_A(I) \cap \Phi_A(J).
    \end{equation}
    Together, \zcref{eq:inclusion-1,eq:inclusion-2} imply that 
    \[
        \Phi_A(I \cdot J) = \Phi_A(I) \cap \Phi_A(J),
    \]
    as required. Finally, the top element is preserved because $\cat L_A = 0$, and so $\Phi_A(A) = \Der(A)$. 

    We have now shown that $\Phi_A$ is a frame morphism. To see that it is an embedding, suppose that $\Phi_A(I) = \Phi_A(J)$, or equivalently, $\cat L_I = \cat L_J$. Since $A/I \in \cat L_I$, it follows that $A/I \in \cat L_J$, and so $J \cdot (A/I) = 0$. Therefore, $J \subseteq \operatorname{Ann}_A(A/I)=I$. By symmetry $I \subseteq J$, so $I = J$. 
    \end{proof}
    \begin{Rem}
    The same argument extends to a connective $\mathbb E_\infty$-ring spectrum $R$ and gives a frame embedding
    \[
        \Phi_R \colon \Idem{\pi_0(R)} \longrightarrow \Sm(\Mod(R)),
    \]
    sending $I$ to the kernel of the smashing localization whose local objects are those $R$-modules $M$ satisfying $I \cdot \pi_n(M) = 0$ for every $n \in \bbZ$.
\end{Rem}
    \begin{Rem}
        Hebestreit--Scholze \cite{hebestreit2024note} give an explicit description of the smashing localization associated to an idempotent ideal $I \subseteq A$. More specifically, they construct a connective derived $A$-algebra $A/I^\infty$ such that restriction of scalars identifies $\Mod(A/I^\infty)$ with $\cat L_I^\infty$. Thus the associated localization endofunctor is given by $M \longmapsto (A/I^\infty) \otimes_A^{\mathbb L} M$. When $I$ is flat as an $A$-module, $A/I^\infty$ agrees with the ordinary quotient $A/I$ \cite[Remark 4(4)]{hebestreit2024note}. This recovers the classical flat idempotent ideal setting of almost ring theory developed by Gabber--Ramero \cite[Section 2.1.1 and Remark 2.1.4(i)]{GabberRamero2003Almost}. Their framework more generally allows $I \otimes_A I$ to be flat, without requiring $I$ itself to be flat. Our proof of \zcref{thm:frame-embedding} does not use this more detailed description.

        We note that the construction of $A/I^\infty$ in \cite{hebestreit2024note} as the totalization of the Amitsur complex of $A \to A/I$ makes sense for arbitrary ideals $I$, although the resulting map need not be a smashing localization. For example, when $A = \bbZ$ and $I = (p)$ for a prime $p$, it gives $\bbZ/(p)^\infty \simeq \bbZ_p$. Indeed, in this case the construction identifies with $\operatorname{holim}_n \bbZ/p^n$.
    \end{Rem}
    \begin{Rem}
        The embedding $\Phi_A$ can also be deduced from recent work of Liang \cite{Liang2026Generalized}. Indeed, \cite[Theorem 7.8]{Liang2026Generalized} gives an equivalence of frames
        \[
        \Idem{A} \simeq \operatorname{cIdem}(\Der_{\infty}(A)_{\ge 0}),
        \]
        where, for a presentably symmetric monoidal $\infty$-category $\cat C$, $\operatorname{cIdem}(\cat C)$ denotes the frame of coidempotent objects in $\cat C$. Moreover, \cite[Proposition 3.9]{Liang2026Generalized} shows that there is a frame embedding
        \[
        \operatorname{cIdem}(\Der_{\infty}(A)_{\ge 0}) \hookrightarrow \operatorname{cIdem}(\Der_{\infty}(A)).
        \]
    By stability, the latter frame is isomorphic to $\Sm(\Der_{\infty}(A))$; see \cite{BalmerFavi2011Generalized} or \cite[Section 2]{Aoki2025Sheaves}. Liang's proof uses the full strength of \cite{hebestreit2024note}. 
    \end{Rem}
    \subsection{Examples of spatial frames of idempotent ideals}
        Combining \zcref{thm:frame-embedding,lem:subframe-of-spatial-frame}, we see that if the frame $\Idem{A}$ is not spatial, then neither is $\Sm(\Der(A))$. On the other hand, in many cases, the frame $\Idem{A}$ is spatial, as we now discuss. For this, we recall that the poset of radical ideals in $A$ forms a spatial frame $\RadId(A)$ isomorphic to $\mathcal{O}(\Spec(A))$, see \cite[Section 1.3]{KockPitsch2017Hochster} for example. Here the join of a collection of ideals $\{ I_{\alpha} \}_{\alpha}$ is given by $\sqrt{\sum_{\alpha}I_{\alpha}}$ while the meet of $I$ and $J$ is $\sqrt{IJ}$. 
    \begin{Lem}
        The assignment $I \mapsto \sqrt{I}$ defines a morphism $\rho \colon \Idem{A} \to \RadId(A)$ of frames. It is injective if and only if
        \begin{equation}\label{eq:spatial-condition}
            \sqrt{I} = \sqrt{J} \implies I = J
        \end{equation}
        for all $I,J \in \Idem{A}$, in which case $\Idem{A}$ is spatial.
    \end{Lem}
    \begin{proof}
        For any family $\{I_\alpha\}_\alpha$ of idempotent ideals, we have
        \[
            \sum_\alpha I_\alpha\subseteq\sum_\alpha\sqrt{I_\alpha}\subseteq\sqrt{\sum_\alpha I_\alpha}.
        \]
        Taking radicals gives $\sqrt{\sum_\alpha I_\alpha}=\sqrt{\sum_\alpha\sqrt{I_\alpha}}$, which says precisely that $\rho$ preserves arbitrary joins. It preserves binary meets because $\sqrt{IJ}=\sqrt I\cap\sqrt J$, and it preserves the top element because $\sqrt A=A$. Thus $\rho$ is a morphism of frames. Finally, \zcref{eq:spatial-condition} is precisely the assertion that $\rho$ is injective, in which case \zcref{lem:subframe-of-spatial-frame} implies that $\Idem{A}$ is spatial, since $\RadId(A)$ is spatial.
    \end{proof}
    \begin{Prop}\label{prop:E(A)-spatial}
        Suppose that $A$ is one of the following:
        \begin{enumerate}
            \item a noetherian ring;
            \item a valuation domain; 
            \item an absolutely flat ring. 
        \end{enumerate}
        Then $\Idem{A}$ is a spatial frame. 
    \end{Prop}
    \begin{proof}
        Assume first that $A$ is noetherian, so that every ideal $I$ is finitely generated. If $I$ is an idempotent ideal, then by Nakayama's lemma, there exists an idempotent $e \in A$ such that $I = (e)$ \cite[\href{https://stacks.math.columbia.edu/tag/00EH}{Tag 00EH}]{stacks-project}. So suppose that $\sqrt{(e)} = \sqrt{(f)}$. Then $e^n \in (f)$ for some $n \ge 1$, but $e$ is idempotent, so $e^n = e$, and therefore $e \in (f)$, and $(e) \subseteq (f)$. By symmetry, $(f) \subseteq (e)$, and hence $(e) = (f)$. Therefore \zcref{eq:spatial-condition} is satisfied for noetherian rings. 
    
        Now suppose that $A$ is a valuation domain. In this case, proper idempotent ideals are automatically prime \cite[Theorem 17.1(3)]{Gilmer1992Multiplicative}, and hence radical. The unit ideal $A$ is radical as well, so $\rho$ is injective.
            
        Finally, if $A$ is absolutely flat, then every ideal is idempotent and radical. Thus $\Idem{A}=\RadId(A)$, so $\rho$ is an isomorphism.
    \end{proof}
    \begin{Rem}
        For a noetherian ring, the frame $\Sm(\Der(A))$ identifies with $\mathcal{O}(\Spec(A)^{\vee})$ by Neeman's classification of smashing subcategories in $\Der(A)$ \cite[Section 3]{Neeman1992Chromatic}. By \zcref{lem:subframe-of-spatial-frame} this gives a roundabout proof that $\Idem{A}$ is spatial! In the case of a finite-dimensional valuation domain, the frame $\Sm(\Der(A))$ is studied in \cite{BalchinTecklenburg2025Classifying}. 
    \end{Rem}
    \begin{Rem}
      Condition \zcref{eq:spatial-condition} is sufficient, but not necessary, for $\Idem{A}$ to be spatial. Let $k$ be a field, let $B=k[t,t^{1/2},t^{1/4},\ldots]$, and let $\mathfrak n=(t^{1/2^n}\mid n\ge 0)$. Then $V=B_{\mathfrak n}$ is a valuation domain with value group $\mathbb Z[1/2]$ and maximal ideal $\mathfrak m=(t^q\mid q\in\mathbb Z[1/2]_{>0})$. This is Keller's valuation domain, with $\ell = 2$, from his counterexample to the generalized smashing conjecture \cite[Section 2]{Keller1994Remark}; see also \cite[Example 5.24]{BazzonivStovivcek2017Smashing}. Notice that $\mathfrak m^2=\mathfrak m$, since $t^q=t^{q/2}t^{q/2}$ for every $q>0$.
    
        Set $A=V/(t)$ and $\overline{\mathfrak m}=\mathfrak m/(t)$. Equivalently, $A\cong k[t^{1/2^\infty}]/(t)$. Since $A$ is a quotient of a valuation domain, its ideals are linearly ordered by inclusion, and hence $\Idem{A}$ is a complete chain. We note that every complete chain is spatial; for each $b<1$, the map
        \[
            p_b(x)=
            \begin{cases}
                0 & \text{if }x\le b,\\
                1 & \text{otherwise}
            \end{cases}
        \]
        is a point. If $a\nleq b$, then $p_b(a)=1$ and $p_b(b)=0$, so $\Idem{A}$ has enough points.\footnote{One can show that $\Idem{A}=\{0<\overline{\mathfrak m}<A\}$. Since we do not need this stronger statement, we do not include the proof.}
    
         However, $\overline{\mathfrak m}$ is a nonzero idempotent ideal of $A$. It is nonzero because $t^{1/2}\notin(t)$, and it is idempotent because $\mathfrak m^2=\mathfrak m$. Moreover, $\overline{\mathfrak m}$ is generated by the elements $t^q$ with $q>0$, each of which is nilpotent in $A$, since $(t^q)^N\in(t)$ whenever $Nq\ge 1$. As the nilpotent elements of a commutative ring form an ideal, $\overline{\mathfrak m}$ is a nil ideal. Since $A$ is local with maximal ideal $\overline{\mathfrak m}$, it follows that
        \[
            \sqrt{0}=\overline{\mathfrak m}=\sqrt{\overline{\mathfrak m}}.
        \]
        Thus the radicalization map $\rho \colon \Idem{A} \to \RadId(A)$ is not injective, even though $\Idem{A}$ is spatial.
    \end{Rem}
    \section{Preliminaries on Boolean algebras}\label{sec:boolean}
    
    \subsection{Boolean algebras and Stone duality}
    \begin{Def}
        A \emph{Boolean algebra} $\ba{B}$ is a bounded distributive lattice in which every element has a complement, i.e., for all $b \in \ba{B}$ there exists an element $\neg b$ with $b \meet \neg b = 0$ and $b \join \neg b = 1$. 
    \end{Def}
    
    \begin{Exa}\label{exa:powerBA}
        If $X$ is a set, then the power set $\ba{P}(X)$ on $X$ naturally carries the structure of a Boolean algebra with joins given by unions and meets given by intersections of subsets. Any finite Boolean algebra is of the form $\ba{P}(X)$ for some finite set $X$.
    \end{Exa}
    
    \begin{Rem}
        Boolean algebras form a category $\BAlg$, where a morphism of Boolean algebras is a morphism of distributive lattices that preserves the complementation.  The category is bicomplete, i.e., admits all (set-indexed) limits and colimits. Of particular interest to us is the coproduct $\oplus$ of two Boolean algebras, which in the literature is also referred to as the \emph{free product}. Under Stone duality (see \zcref{thm:stone-duality}), it corresponds to the Cartesian product: Given $\ba{B}_1$ and $\ba{B}_2$, we have
            \[
                \ba{B}_1 \oplus \ba{B}_2 \cong \Clop(\Spec(\ba{B}_1) \times \Spec(\ba{B}_2)).
            \]
    \end{Rem}
    \begin{Rem}\label{rem:booleanrings}
        Boolean algebras may equivalently be described in terms of \emph{Boolean rings}, i.e., commutative rings $B$ in which $x^2=x$ for every $x \in 
        B$. Write $\BRings$ for the full subcategory of the category of commutative rings on the Boolean rings. Then there is an isomorphism of categories
            \[
                \BAlg \xlongrightarrow{\sim} \BRings
            \]
        induced by sending a Boolean algebra $\ba{B}$ to the Boolean ring $B$ with the same underlying set and operations $a \cdot b = a \meet b$ and $a + b = (a \meet \neg b) \join (\neg a \meet b)$ for $a,b \in B$. Note that Boolean rings are automatically $\bbF_2$-algebras and that the coproduct of Boolean algebras corresponds to the tensor product (over $\bbF_2)$ of the corresponding Boolean rings. 
    \end{Rem}
    \begin{Exa}
        Let $X$ be a non-empty topological space, and let $\Clop(X)$ denote the set of clopen sets in $X$. Then $\Clop(X)$ is a Boolean algebra; the join and meet are given by union and intersection respectively, and $\neg U \coloneqq X \setminus U$. 
    \end{Exa}
    In fact, every Boolean algebra is isomorphic to one arising in this way. We recall that a Stone space is a compact Hausdorff totally disconnected space.
    \begin{Thm}[Stone duality]\label{thm:stone-duality}
        Every Boolean algebra $\ba{B}$ is isomorphic to $\Clop(X)$ for some Stone space $X$. More precisely, there is an equivalence of categories
        \[
            \BAlg^{\mathrm{op}} \simeq \mathbf{Stone}.
        \]
    \end{Thm}
    \begin{Rem}
        Given a Boolean algebra $\ba{B}$, the associated Stone space is $X_{\ba{B}} \coloneqq \Spec(\ba{B})$. Concretely, the set underlying $\Spec(\ba{B})$ is given by the set $\Ult(\ba{B})$  of ultrafilters of $\ba{B}$. This set comes equipped with the Stone topology, with a basis of opens formed by the sets $\{\cat F \in \Ult(\ba{B}) \mid b \in \cat F\}$ where $b$ runs through the elements of $\ba{B}$.
    \end{Rem}
    \begin{Exa}\label{exa:stonecech}
        Continuing \zcref{exa:powerBA}, let $X$ be a set and write $\ba{P}(X)$ for the Boolean algebra of subsets of $X$. The Stone representative of $\ba{P}(X)$ can then be identified as
            \[
                \Spec(\ba{P}(X)) \cong \beta X
            \]
        where $\beta X$ is the Stone--\v{C}ech compactification of $X$ viewed as a discrete space.
    \end{Exa}
    \begin{Exa}
        Given a set $S$ there exists a \emph{free Boolean algebra} $\ba{F}(S)$, for example by observing that the forgetful functor from Boolean algebras to sets admits a left adjoint. Under Stone duality, the corresponding space is $2^S$, and $\ba{F}(S) \cong \Clop(2^S)$. For $n \geq 0$ a natural number, the free Boolean algebra $\ba{F}_{[n]}$ on $n$ generators identifies with $\ba{P}\ba{P}([n])$, where $[n]\coloneqq \{0,1,\ldots,n-1\}$ denotes an $n$ element set. In particular, $\ba{F}_{[n]}$ has $2^{2^n}$ elements.
    \end{Exa}
    \begin{Def}
        Let $\ba{F}_\omega \coloneqq \ba{F}(\mathbb{N})$ denote the free Boolean algebra on countably many generators.
    \end{Def}
    
    \begin{Rem}\label{rem:coproductFw}
        We have $\ba{F}_{\omega} = \bigcup_{n\geq 0} \ba{F}_{[n]} \cong \colim_{n \geq 0}\ba{P}\ba{P}([n])$. Here, the transition map $\ba{P}\ba{P}([n]) \to \ba{P}\ba{P}([n+1])$ is induced by applying the (contravariant) power set functor twice to the inclusion $[n]=\{0,1,\ldots,n-1\} \to \{0,1,\ldots,n\} =  [n+1]$. Under Stone duality, this presentation of $\ba{F}_{\omega}$ corresponds to the usual construction of Cantor space $2^{\bbN} \cong \lim_n 2^{[n]}$. Moreover, since the free functor $\ba{F}$ is a left adjoint, any bijection $\mathbb{N} \amalg \mathbb{N} \cong \mathbb{N}$ induces an isomorphism $\ba{F}_{\omega} \oplus \ba{F}_{\omega} \xrightarrow{\sim} \ba{F}_{\omega}$ of Boolean algebras. 
    \end{Rem}
    
    We finish with one final definition. 
    
    \begin{Def}\label{def:principal-ideal}
        Let $\ba{B}$ be a Boolean algebra. For any nonzero $b \in \ba{B}$, set
        \[
            \ba{B}_b \coloneqq \Set{x \in \ba{B}\given x \leq b},
        \]
        called the \emph{principal ideal} of $\ba{B}$ generated by $b$. Then $\ba{B}_b$ is a Boolean algebra whose meet and join are inherited from $\ba{B}$, whose bottom element is $0$, whose top element is $b$, and whose complement operation is given by $\neg_b x \coloneqq b \meet \neg x$.
    \end{Def}
    
    \subsection{Atomless Boolean algebras}
    
    \begin{Hyp}\label{hyp:nontrivial}
        For the remainder of this paper, we always assume that $0 \ne 1$ in a Boolean algebra, i.e., we exclude the trivial Boolean algebra.
    \end{Hyp}
    
    \begin{Def}\label{def:atoms}
        An \emph{atom} in a Boolean algebra $\ba{B}$ is a minimal nonzero element $a \in \ba{B}$. We write $\At(\ba{B})$ for the set of atoms in $\ba{B}$. A Boolean algebra is called \emph{atomless} if $\At(\ba{B}) = \emptyset$.
    \end{Def}
    
    \begin{Exa}\label{exa:atomsFn} 
        For any set $X$, the atoms of the power set Boolean algebra $\ba{P}(X)$ are precisely given by the singletons, i.e., $\At(\ba{P}(X)) = X$. In particular, the atoms of the free Boolean algebra $\ba{F}_{[n]}$ on $n$ generators are given by $\At(\ba{F}_{[n]}) \cong \At(\ba{P}\ba{P}([n])) \cong \ba{P}([n])$. Explicitly, if we write $g_0,\ldots,g_{n-1}$ for the generators of $\ba{F}_{[n]}$, then the atoms are $\bigwedge_{i=0}^{n-1} g_i^{\epsilon_i}$ running through strings  $(\epsilon_i) \in \{1,-1\}^{[n]}$. Here, we write $x^{-1} \coloneqq \neg x$ for any $x \in \ba{F}_{[n]}$.  
    \end{Exa}
    \begin{Rem}\label{rem:isolated-points}
        A Boolean algebra $\ba{B}$ is atomless if and only if the associated Stone space $X_{\ba{B}}$ is \emph{perfect}, i.e., it does not contain any isolated points (i.e., singleton opens). 
    \end{Rem}    
    \begin{Rem}\label{rem:atomlessinfinity}
        Any atomless Boolean algebra must be infinite. Moreover, the coproduct of two nontrivial Boolean algebras is atomless if and only if at least one of them is atomless. Indeed, under Stone duality, coproducts of Boolean algebras correspond to products of Stone spaces, and a product of two nonempty spaces has an isolated point if and only if both factors have an isolated point.
    \end{Rem}   
    \begin{Lem}\label{lem:atomlessprincipalideal}
        Let $\ba{B}$ be an atomless Boolean algebra. For any nonzero $b \in \ba{B}$, the principal ideal $\ba{B}_b = \{x \in \ba{B}\mid x \leq b\}$ viewed as a Boolean algebra is also atomless. 
    \end{Lem}
    \begin{proof}
        Because $\ba{B}$ is atomless, for any nonzero $x \in \ba{B}_{b}$, there exists some $0 < y < x$ in $\ba{B}$. By definition $y \in \ba{B}_{b}$, so $x$ is not an atom in $\ba{B}_b$ either. 
    \end{proof}
    
    \begin{Def}
        Two subalgebras $\ba{A}_1,\ba{A}_2$ of a Boolean algebra $\ba{B}$ are said to be \emph{independent} if $a_1 \meet a_2 >0$ for all nonzero $a_1 \in \ba{A}_1$ and $a_2 \in \ba{A}_2$. Equivalently, by \cite[Corollary 44.1]{GivantHalmos2009Introduction}, the canonical homomorphism
        \[
            \ba{A}_1\oplus\ba{A}_2  \longrightarrow \langle\ba{A}_1,\ba{A}_2\rangle_{\ba{B}}
        \]
        is an isomorphism, where the target denotes the smallest Boolean subalgebra of $\ba{B}$ containing $\ba{A}_1$ and $\ba{A}_2$.
    \end{Def}
    
    \begin{Thm}\label{thm:omegafree}
        The free Boolean algebra $\ba{F}_{\omega}$ on a countably infinite set of generators satisfies the following properties:
        \begin{enumerate}
            \item $\ba{F}_{\omega}$ is atomless;
            \item any countable atomless Boolean algebra is isomorphic to $\ba{F}_{\omega}$. 
        \end{enumerate}
    \end{Thm}
    \begin{proof}
       Part (1) is shown in \cite[Corollary 28.2]{GivantHalmos2009Introduction}, and Part (2) then follows from \cite[Theorem 10.2]{GivantHalmos2009Introduction}. 
    \end{proof}
    
\begin{Exa}\label{exa:atomlessBA}
    Atomless Boolean algebras exist in profusion. They can easily be constructed using perfect Stone spaces, i.e., those without isolated points; see \zcref{rem:isolated-points}. Let $\kappa$ be an infinite cardinal. Under Stone duality, the free Boolean algebra $\ba{F}_\kappa$ on $\kappa$ generators corresponds to the Cantor cube $2^\kappa$, and hence $\ba{F}_\kappa \cong \Clop(2^\kappa)$. Since every basic open subset of $2^\kappa$ restricts only finitely many coordinates, the space $2^\kappa$ has no isolated points. It follows that $\ba{F}_\kappa$ is atomless.

    Another example arises from the Stone--\v{C}ech compactification (\zcref{exa:stonecech}) of a discrete space $X$. Viewing $X$ as a subset of $\beta X$ via the principal ultrafilters on the points of $X$, the complement $\beta X \setminus X$ is a perfect Stone space. The associated atomless Boolean algebra can be identified with $\ba{P}(X)/\mathrm{Fin}$, where $\mathrm{Fin}$ is the ideal generated by the singletons of $X$, see \cite[Example~5.28(b)]{Koppelberg1989Handbook}.
\end{Exa}
    
    The key property of atomless Boolean algebras we will use in the next section is the abundance of independent atomless subalgebras:
    
    \begin{Prop}\label{prop:splittingprinciple}
        Every atomless Boolean algebra contains a pair of independent atomless subalgebras.
    \end{Prop}
    \begin{proof}
        The first step is to find a countable atomless Boolean subalgebra of $\ba{B}$. Indeed, proceeding inductively, we begin by constructing a sequence of finite subalgebras $\ba{B}_0 \subset \ba{B}_1 \subset \ba{B}_2 \subset \ba{B}_3 \subset \ldots \subset \ba{B}$. Let $\ba{B}_0 = \{0,1\}$ and suppose that the sequence has been constructed up to $\ba{B}_n$. For each $a \in \At(\ba{B}_n)$, atomlessness of $\ba{B}$ gives some $c_a \in \ba{B}$ with $0 < c_a < a$. Let $\ba{B}_{n+1}$ be the finite Boolean subalgebra of $\ba{B}$ generated by $\ba{B}_{n}$ and $\{c_a\mid a \in \At(\ba{B}_n)\}$. Finally, we set
            \[
                \ba{A} \coloneqq \bigcup_{i \geq 0}\ba{B}_i \subseteq \ba{B}.
            \]
        By construction, $\ba{A}$ is a countable Boolean subalgebra of $\ba{B}$. It is also atomless: Indeed, any nonzero $x \in \ba{A}$ comes from a finite stage $x \in \ba{B}_n$, thus is not an atom in $\ba{B}_{n+1}$ and hence not in $\ba{A}$ either.
    
        Combining \zcref{thm:omegafree} and \zcref{rem:coproductFw}, we obtain a (non-canonical) isomorphism $\theta \colon \ba{F}_{\omega} \oplus \ba{F}_{\omega} \xrightarrow{\sim} \ba{A}$. Let $\iota_i\colon\ba{F}_{\omega}\hookrightarrow\ba{F}_{\omega}\oplus\ba{F}_{\omega}$ denote the canonical inclusions, and set $\ba{A}_i\coloneqq\theta(\iota_i(\ba{F}_{\omega}))$ for $i=1,2$. Each $\ba{A}_i$ is isomorphic to $\ba{F}_{\omega}$, and is therefore countable and atomless. Moreover, the canonical morphism
        \[
        \ba{A}_1\oplus\ba{A}_2\longrightarrow\langle\ba{A}_1,\ba{A}_2\rangle_{\ba{A}}=\ba{A}
        \]
        identifies with $\theta$ and is therefore an isomorphism. Hence $\ba{A}_1$ and $\ba{A}_2$ are independent. Since $\ba{A}\subseteq\ba{B}$, they are also independent as subalgebras of $\ba{B}$.
    \end{proof}
    
    \begin{Rem}
        Under Stone duality, an embedding $\ba{F}_{\omega} \hookrightarrow \ba{B}$ corresponds to a continuous surjection $X_{\ba{B}} \twoheadrightarrow 2^{\bbN}$ from the space representing $\ba{B}$ to the Cantor space. Such a map can be constructed by iteratively subdividing clopen subsets $X_{\ba{B}} $ into pairs of non-trivial clopen subsets, which is always possible when $X_{\ba{B}}$ does not contain any isolated points. The two independent subalgebras correspond geometrically to the splitting $2^{\bbN} \cong 2^{\bbN} \times 2^{\bbN}$, e.g., obtained by separating odd and even indices in a string representing elements of Cantor space.
    \end{Rem}
    
    \section{Nil-Cantor rings}\label{sec:rings}

    For simplicity, throughout most of this section we restrict to coefficients in a field $k$; however, the constructions provided below and their essential features generalize to any nonzero coefficient commutative ring.

    \subsection{Construction}\label{ssec:construction}
    
    We first construct a functor on \emph{finite} Boolean algebras and injective homomorphisms taking values in the category of commutative $k$-algebras. If $\ba{B}$ is a finite Boolean algebra, set $\R{\ba{B}} = \bigoplus_{b \in \ba{B}}ke_b$ as a $k$-vector space, equipped with the multiplication extended $k$-linearly from the formula
        \begin{equation}\label{eq:multiplication}
            e_b\cdot e_c = 
                \begin{cases}
                    e_{b \join c} & \text{if } b \meet c = 0 \\
                    0 & \text{otherwise},
                \end{cases}
        \end{equation}
    for any $b,c \in \ba{B}$. The unit in $\R{\ba{B}}$ is $1 = e_0$ and associativity follows from distributivity in $\ba{B}$. Since $b = \bigvee_{a \in \At(\ba{B}), a \leq b}a$, we deduce that $e_b = \prod_{a \in \At(\ba{B}), a \leq b}e_a$. If we denote the atoms in $\ba{B}$ by $\{a_1,\ldots,a_n\} = \At(\ba{B})$, then there is a canonical isomorphism
        \begin{equation}\label{eq:finite-ring}
            k[x_1,\ldots,x_n]/(x_i^2) \xlongrightarrow{\sim} \R{\ba{B}},
    \end{equation}
    sending $x_i$ to $e_{a_i}$. If $k = \bbF_2$, this identifies with an exterior algebra on $n$ generators. An injective homomorphism $f\colon \ba{B}_1 \to \ba{B}_2$ induces a map $\R{f}\colon \R{\ba{B}_1} \to \R{\ba{B}_2}$ of $k$-algebras determined by the assignment $e_b \mapsto e_{f(b)}$ for any $b \in \ba{B}_1$; here, injectivity is required for $\R{f}$ to be a homomorphism of $k$-algebras.
    
    This construction extends to arbitrary Boolean algebras and injective homomorphisms between them, as follows: Any Boolean algebra $\ba{B}$ can be written as a union $\ba{B} = \bigcup_{i \in I}\ba{B}_i$ of its finite Boolean subalgebras $\ba{B}_i \subseteq \ba{B}$. Applying $\R{-}$ to the corresponding directed diagram, we can extend $\R{-}$ to any Boolean algebra: 
        \begin{equation}\label{eq:colimitring}
            \R{\ba{B}} \coloneqq \colim_{i \in I} \R{\ba{B}_i}.
        \end{equation}
    More explicitly, $\R{\ba{B}}$ has a $k$-basis given by $\{e_{b}\mid b \in \ba{B}\}$ with multiplication determined again by \zcref{eq:multiplication}. In particular, $e_0$ is the unit in $\R{\ba{B}}$.

    \begin{Def}\label{def:nilCantorring}
        Given a Boolean algebra $\ba{B}$, we call $\R{\ba{B}}$ the \emph{nil-Cantor ring}\footnote{This terminology was suggested to us by Dave Benson.} of $\ba{B}$.
    \end{Def}
    
    \begin{Rem}
        A related construction appears in \cite[p.~50]{CrapoSchmitt2008Primitive}. Let $\cat U$ denote the set of finite subsets of an infinite set $X$. Crapo--Schmitt define $k\{ \cat U \}$ to be the $k$-algebra with basis $\cat U$, and multiplication
        \[
        ST = \begin{cases}
            S \cup T, & \text{ if } S \cap T = \varnothing, \\
            0 & \text{ otherwise},
        \end{cases}
        \]
        for all $S,T \in \cat U$. In particular, 
        \[
        k\{ \cat U\} \cong k[x_i\mid i \in X]/(x_i^2)
        \]
        where $x_i$ corresponds to the basis element indexed by the singleton $\{i \} \in \cat U$. 
    \end{Rem}
    
    \begin{Exa}\label{exa:HSring}
        Let us unpack the construction of $\R{\ba{B}}$ for the simplest nonzero example of atomless $\ba{B}$, namely $\ba{B} = \ba{F}_{\omega}$. By \zcref{rem:coproductFw}, we have  $\ba{F}_{\omega} = \colim_n\ba{P}\ba{P}([n])$ with transition maps induced by the natural inclusions $[n] \to [n+1]$. Using \zcref{exa:atomsFn}, the atoms of $\ba{F}_{[n]} \cong \ba{P}\ba{P}([n])$ are indexed on subsets $S \in \ba{P}([n])$. Therefore, via the isomorphism of \zcref{eq:finite-ring} specialized to $\ba{F}_{[n]}$, the transition maps are given by
    \[\begin{tikzcd}[cramped]
    	{\R{\ba{F}_{[n]}}} & {k[x_S\colon S \in \ba{P}([n])]/(x_S^2)} & {x_S} \\
    	{\R{\ba{F}_{[n+1]}}} & {k[y_{S'}\colon {S'} \in \ba{P}([n+1])]/(y_{S'}^2)} & {y_Sy_{S \cup \{n\}}.}
    	\arrow["\sim", from=1-1, to=1-2]
    	\arrow[from=1-1, to=2-1]
    	\arrow[from=1-2, to=2-2]
    	\arrow[maps to, from=1-3, to=2-3]
    	\arrow["\sim"', from=2-1, to=2-2]
    \end{tikzcd}\]
        Represent $\ba{P}([n])$ as a length $n$ string of $0$s and $1$s and write $2^{<\bbN}$ for the set of finite strings of 0's and 1's. Taking the colimit over $n$ in the above presentation then induces an isomorphism
            \[  
                \R{\ba{F}_{\omega}} \cong k[T_S \mid S \in 2^{<\bbN}]/(T_S-T_{S \ast 0} \cdot T_{S \ast 1},T_S^2 \mid S \in 2^{<\bbN}),
            \]
        where $S\ast i$ denotes the concatenation of $S$ with $i\in\{0,1\}$. Consequently, for $k = \bbF_2$, the ring $\R{\ba{F}_{\omega}}$ coincides with the third ring described in the introduction of \cite{hebestreit2024note}.
    \end{Exa}
    \begin{Lem}\label{lem:localring}
        There is an augmentation $\varepsilon\colon \R{\ba{B}} \to k$ whose kernel $\mathfrak{m}_{\ba{B}}\coloneqq \ker(\varepsilon)$ is the unique prime ideal in $\R{\ba{B}}$. In particular, $\R{\ba{B}}$ is local.
    \end{Lem}
    \begin{proof}
        Define the map $\varepsilon\colon \R{\ba{B}} \to k$ on basis elements by setting
            \[
                \varepsilon(e_b) = 
                    \begin{cases}
                        0 & \text{if } b \neq 0; \\
                        1 & \text{if } b=0
                    \end{cases}
            \]
        and then extending $k$-linearly. We see that $\varepsilon$ provides a surjective $k$-algebra homomorphism, i.e., it is an augmentation. Consequently, the kernel is given by the ideal
            \begin{equation}\label{eq:maxideal}
                \mathfrak{m}_{\ba{B}} = \ker(\varepsilon) = \langle \{e_b \mid b \in \ba{B}\setminus\{0\}\}\rangle.
            \end{equation}
        Any element $x \in \mathfrak{m}_{\ba{B}}$ comes from some finite-dimensional subalgebra $\R{\ba{B}_i}$, and thus will be nilpotent of exponent $\dim_{k}(\R{\ba{B}_i})+1$. Therefore, $\mathfrak{m}_{\ba{B}}$ coincides with the nil-radical of $\R{\ba{B}}$, so $\R{\ba{B}}$ has unique prime ideal $\mathfrak{m}_{\ba{B}}$.
    \end{proof}
    
    \begin{Rem}\label{rem:localringalt}
        Alternatively, \zcref{lem:localring} also follows directly from \zcref{eq:colimitring}. In fact, continuity of the Zariski spectrum together with \zcref{eq:finite-ring} imply that 
            \[
                \Spec(\R{\ba{B}}) = \{\mathfrak{m}_{\ba{B}}\}.
            \]
        It thus follows from the Hopkins--Neeman--Thomason theorem \cite{Hopkins1987,Neeman1992Chromatic,Thomason1997} that $\Spc(\Der(\R{\ba{B}})^{\omega}) \cong \{\ast\}$. 
    \end{Rem}
    
    \begin{Lem}\label{lem:tangent}
        For any Boolean algebra $\ba{B}$, we have an isomorphism of $k$-vector spaces
            \[
                \mathfrak{m}_{\ba{B}}/\mathfrak{m}_{\ba{B}}^2 \cong \bigoplus_{a \in \At(\ba{B})}k\overline{e_a},
            \]
        where $\overline{e_a}$ denotes the image of $e_a$ under the evident quotient map. In particular, $\ba{B}$ is atomless if and only if $\mathfrak{m}_{\ba{B}} \in \Idem{\R{\ba{B}}}$.
    \end{Lem}
    \begin{proof}
        We first prove the auxiliary claim that, as $k$-vector spaces, we have
            \[
                \mathfrak{m}_{\ba{B}}^2 = \bigoplus_{ 0 \neq b \notin \At(\ba{B})}ke_b.
            \]
        The nonzero products $e_c \cdot e_d$ for nonzero $c,d$ must be of the form $e_b$ for $b = c \join d \neq 0$ and $c \meet d = 0$. In particular, $b$ cannot be an atom. To see the reverse inclusion, consider some nonzero $b \notin \At(\ba{B})$, so there exists some $c \in \ba{B}$ with $0 < c < b$. Set $d = b \meet \lnot c$. This implies that both $c$ and $d$ are nonzero, $c\meet d = 0$, and $c \join d = b$. Therefore, we have $e_b = e_c\cdot e_d$, as desired.
    
        The description of the quotient $\mathfrak{m}_{\ba{B}}/\mathfrak{m}_{\ba{B}}^2$ is then an immediate consequence, and the characterization of atomlessness of $\ba{B}$ follows.
    \end{proof}
    
    \subsection{Idempotent ideals}\label{ssec:idempotentideals}
    
    Let $\ba{B}$ be an atomless Boolean algebra and fix some nonzero $b \in \ba{B}$. For every atomless subalgebra $\ba{A}$ of $\ba{B}_b$ let 
        \begin{equation}\label{eq:idealmachine}
            \ideal{b}(\ba{A}) \coloneqq \langle \{ e_a \mid 0 \neq a \in \ba{A}\}\rangle_{\R{\ba{B}}}
        \end{equation}
    be the ideal in $\R{\ba{B}}$ generated by the elements $e_a$ for $a$ running through the nonzero elements of $\ba{A}$. A $k$-vector space basis for $\ideal{b}(\ba{A})$ in $\R{\ba{B}}$ is given by the $e_x$ for all $x$ in the upward closure of $\ba{A}\setminus\{0\}$ inside $\ba{B}$. Indeed, the span of these basis elements is an ideal. Conversely, if $0 < a \leq x$ with $ a \in \ba{A}$, then $e_x=e_a \cdot e_{x\meet\neg a}$, so $e_x$ belongs to the ideal generated by the $e_a$. As a partial generalization of \zcref{lem:tangent}, this produces a good supply of idempotent ideals in $\R{\ba{B}}$:
    
    \begin{Lem}\label{lem:idempotentidealmachine}
        With notation as above, $\ideal{b}(\ba{A})$ is an idempotent ideal in $\R{\ba{B}}$.
    \end{Lem}
    \begin{proof}
        Consider some nonzero $a \in \ba{A}$. Since $\ba{A}$ is atomless, there exist nonzero  $a_1,a_2 \in \ba{A}$ with $a_1 \meet a_2 = 0$ and $a_1 \join a_2 = a$. Consequently, $e_a = e_{a_1}\cdot e_{a_2}$. This gives the inclusion $\ideal{b}(\ba{A}) \subseteq \ideal{b}(\ba{A})^2$, while the reverse inclusion holds unconditionally.
    \end{proof}
    
    \begin{Exa}
        Taking $b=1$ and $\ba{A} = \ba{B}$ itself, we get back $\ideal{1}(\ba{B}) = \mathfrak{m}_{\ba{B}}$.
    \end{Exa}
    
    \begin{Lem}\label{lem:idempotentidealmachine2}
        For any atomless Boolean algebra $\ba{B}$, the assignment $0 \mapsto (0)$ and $b \mapsto \ideal{b}(\ba{B}_b)$ for $b \ne 0$ provides an order embedding $\ba{B} \hookrightarrow \Idem{\R{\ba{B}}}$. In particular, the set of idempotent ideals of $\R{\ba{B}}$ has at least the cardinality of  $\ba{B}$.
    \end{Lem}
    \begin{proof}
        Any nonzero principal ideal in an atomless Boolean algebra is itself atomless (\zcref{lem:atomlessprincipalideal}), so \zcref{lem:idempotentidealmachine} guarantees that the map indeed takes values in idempotent ideals of $\R{\ba{B}}$. The claims concerning the zero element are immediate, so it remains to consider nonzero elements $b,c\in\ba{B}$. By our above discussion, a basis element $e_x \in \ideal{b}(\ba{B}_b)$ if and only if $0 < x \meet b$. Consequently, we have $\ideal{b}(\ba{B}_b) \subseteq \ideal{c}(\ba{B}_c)$ if and only if $x \meet b \neq 0$ implies $x \meet c \neq 0$ for all $x \in \ba{B}$.
    
        If $b \leq c$, then $b \meet x \leq c \meet x$ for any $x \in \ba{B}$, which shows that our map is order preserving. In order to see that it also reflects the order, assume $\ideal{b}(\ba{B}_b) \subseteq \ideal{c}(\ba{B}_c)$. Since $\lnot c \meet c = 0$, the criterion above gives $\lnot c \meet b = 0$, which in turn implies $b = (\lnot c \join c) \meet b = c \meet b$. This gives $b \leq c$, as claimed.
    \end{proof}
    
    \begin{Lem}\label{lem:orthogonalideals}
        Suppose $\ba{B}$ is an atomless Boolean algebra. For each nonzero $b \in \ba{B}$ there exist ideals $J_b(1), J_b(2) \in \Idem{\R{\ba{B}}}$ such that
            \begin{enumerate}
                \item $e_b \in J_b(i)$ for $i=1,2$; 
                \item $J_b(1)J_b(2) = 0$.
            \end{enumerate}
    \end{Lem}
    \begin{proof}
        Fix some nonzero $b \in \ba{B}$. By \zcref{lem:atomlessprincipalideal}, $\ba{B}_b$ is atomless, so there exist two independent atomless Boolean subalgebras $\ba{A}_1$ and $\ba{A}_2$ inside $\ba{B}_b$ thanks to \zcref{prop:splittingprinciple}. Applying the construction from \zcref{eq:idealmachine} and using \zcref{lem:idempotentidealmachine} furnishes two idempotent ideals $J_b(i)\coloneqq \ideal{b}(\ba{A}_i)$ in $\R{\ba{B}}$ for $i=1,2$. By construction, they both contain $e_b$, so it remains to verify the orthogonality claim (2).
    
        To this end, consider two basis elements $e_{x_i} \in J_b(i)$. By the description of the basis of $J_b(i)$ given above, this means that there exist $a_i \in \ba{A}_i$ such that $0 < a_i \leq x_i$. Therefore, independence gives
            \[
                0 \neq a_1 \meet a_2 \leq x_1 \meet x_2.
            \]
        The definition of the product in $\R{\ba{B}}$ translates this property into $e_{x_1}\cdot e_{x_2} = 0$, so $J_b(1)$ is orthogonal to $J_b(2)$.
    \end{proof}
    
    \subsection{Prime idempotent ideals}\label{ssec:points}
    
    \begin{Prop}\label{prop:idempoints}
        For any atomless Boolean algebra $\ba{B}$, we have 
            \[
                \pt(\Idem{\R{\ba{B}}}) = \{\mathfrak{m}_{\ba{B}}\},
            \]
        where we identify $\mathfrak{m}_{\ba{B}}$ with the point it represents.
    \end{Prop}
    \begin{proof}
        By \zcref{lem:points-prime-elements}, it is enough to show that $\mathfrak m_{\ba B}$ is the unique prime element of $\Idem{\R{\ba B}}$. We first observe that $\mathfrak{m}_{\ba{B}} \in \Idem{\R{\ba{B}}}$ is prime: Let $I,J$ be idempotent ideals with $IJ \subseteq \mathfrak{m}_{\ba{B}}$. If $I \nsubseteq \mathfrak{m}_{\ba{B}}$, it must be $I = \R{\ba{B}}$ by locality (\zcref{lem:localring}), hence $IJ = J \subseteq \mathfrak{m}_{\ba{B}}$. 
    
        Now consider some prime $\mathfrak{p} \in \Idem{\R{\ba{B}}}$ and let $0 \neq b \in \ba{B}$ be some element. Using the second part of \zcref{lem:orthogonalideals}, we get idempotent ideals $J_b(1),J_b(2)$ with $J_b(1)J_b(2) = (0) \subseteq \mathfrak{p}$. Primeness implies that one of them must already be contained in $\mathfrak{p}$, so the first part of \zcref{lem:orthogonalideals} shows $e_b \in \mathfrak{p}$. Since $b$ was an arbitrary nonzero element of $\ba{B}$, we see $\mathfrak{m}_{\ba{B}} \subseteq \mathfrak{p}$, which means they must in fact be equal. In other words, $\mathfrak{m}_{\ba{B}}$ is the unique prime element in the frame $\Idem{\R{\ba{B}}}$.
    \end{proof}
    
    Combining \zcref{rem:atomlessinfinity}, \zcref{lem:idempotentidealmachine2}, and \zcref{prop:idempoints} gives
    
    \begin{Cor}\label{cor:idemnonspatial}
        If $\ba{B}$ is atomless, then $\Idem{\R{\ba{B}}}$ is an infinite frame with a single point, hence cannot be spatial.
    \end{Cor}
    
    Putting everything together, we can now prove our main result. 
    \begin{Thm}\label{thm:main-theorem}
        If $\ba{B}$ is an atomless Boolean algebra, then the frame $\Sm(\Der(\R{\ba{B}}))$ of smashing ideals in the derived category of the nil-Cantor ring $\R{\ba{B}}$ is not spatial. 
    \end{Thm}
    \begin{proof}
        This follows from \zcref{thm:frame-embedding,cor:idemnonspatial}. 
    \end{proof}
    \begin{Exa}\label{exa:large}
        \zcref{thm:main-theorem} applies in particular to the ring discussed in \zcref{exa:HSring}. In fact, by \zcref{thm:omegafree}, this is, up to isomorphism, the only example arising from a countable atomless Boolean algebra. There are many more uncountable examples, such as the ones discussed in \zcref{exa:atomlessBA}. In particular, for any uncountable cardinal $\kappa$, the free Boolean algebra $\ba{F}_\kappa$ on $\kappa \neq 0$ generators is atomless, and the associated ring is an uncountable variant of the ring given by Hebestreit and Scholze. 
    \end{Exa}

    \begin{Rem}\label{rem:moregeneralcoefficients}
    As mentioned at the beginning of this section, the main features of nil-Cantor rings generalize to any nonzero coefficient commutative ring $k$ in place of a field. For the purposes of this remark, we will write $\mathfrak{R}_k(\ba{B})$ for this ring. It has the structure of an augmented $k$-algebra, with augmentation
        \[
            \varepsilon\colon \mathfrak{R}_k(\ba{B}) \to k, \quad 
                \varepsilon(e_b) = 
                    \begin{cases}
                        0 & \text{if } b \neq 0; \\
                        1 & \text{if } b=0,
                    \end{cases}
        \]
    and we again write $\mathfrak{p}_{\ba{B}} \coloneqq \ker(\varepsilon)$. Since the $e_b$ still form a $k$-basis, \zcref{lem:tangent,lem:orthogonalideals} and their proofs carry over verbatim, so the argument of \zcref{prop:idempoints} shows that every prime element of $\Idem{\mathfrak{R}_k(\ba{B})}$ contains $\mathfrak{p}_{\ba{B}}$. Moreover, $\varepsilon$ induces an isomorphism between the idempotent ideals containing $\mathfrak{p}_{\ba{B}}$ and $\Idem{k}$, with inverse $J \mapsto J\mathfrak{R}_k(\ba{B}) + \mathfrak{p}_{\ba{B}}$; the latter is idempotent because $\mathfrak{p}_{\ba{B}}^2 = \mathfrak{p}_{\ba{B}}$. It follows that the map
        \[
            \pt(\varepsilon_*)\colon \pt(\Idem{k})  \xlongrightarrow{\sim} \pt(\Idem{\mathfrak{R}_k(\ba{B})})
        \]
    is a bijection. In particular, no point distinguishes the nonzero idempotent ideal $\mathfrak{p}_{\ba{B}}$ from $0$. We conclude that, for any atomless Boolean algebra $\ba{B}$ and any nonzero commutative ring $k$, the frames $\Idem{\mathfrak{R}_k(\ba{B})}$ and hence also $\Sm(\Der(\mathfrak{R}_k(\ba{B})))$ are not spatial.
\end{Rem}
    
    \begin{Rem}\label{rem:pBoolean}
        Variations of the construction of $\R{\ba{B}}$ give rise to more examples of rings whose frame of idempotent ideals is not spatial. For instance, fix a prime $p$ and let $\ba{B}$ be a finite Boolean algebra. Write $\mathcal{X}^{(p)}(\ba{B}) \coloneqq \Hom(\At(\ba{B}),\{0,\ldots, p-1\})$ for the set of functions on the atoms of $\ba{B}$ with values in $\{0,\ldots, p-1\}$ equipped with the natural partial addition. We then define a commutative $k$-algebra with underlying $k$-vector space
            \[
                \mathfrak{R}^{(p)}(\ba{B}) \coloneqq \bigoplus_{\chi \in \mathcal{X}^{(p)}(\ba{B})}ke_{\chi}
            \]
        and multiplication
            \[
                e_{\chi}\cdot e_{\psi} = 
                \begin{cases}
                    e_{\chi + \psi} & \text{if } \chi(a) + \psi(a) < p \text{ for all } a \in \At(\ba{B}); \\
                    0 & \text{otherwise}.
                \end{cases}
            \]
        As before, we can verify that this indeed defines the structure of a commutative $k$-algebra on $\mathfrak{R}^{(p)}(\ba{B})$. If $\lvert \At(\ba{B})\rvert = n$, then there is an isomorphism 
            \[
                k[x_1,\ldots,x_n]/(x_1^p,\ldots,x_n^p) \xlongrightarrow{\sim} \mathfrak{R}^{(p)}(\ba{B}),
            \]
        induced by sending $x_a$ to $e_{\delta_a}$, where $\delta_a$ denotes the delta-function on the atom $a \in \At(\ba{B})$. With this presentation, an injective map of finite Boolean algebras $f\colon \ba{B}_1 \to \ba{B}_2$ induces a $k$-algebra homomorphism 
            \[
                \mathfrak{R}^{(p)}(f)\colon \mathfrak{R}^{(p)}(\ba{B}_1) \to \mathfrak{R}^{(p)}(\ba{B}_2), \quad x_{a_1} \mapsto \prod_{a_2 \in \At(\ba{B}_2),\, a_2 \leq f(a_1)} x_{a_2}.
            \]
        Finally, writing any Boolean algebra $\ba{B} = \bigcup_{i \in I}\ba{B}_i$ as a union of its finite Boolean subalgebras, we define
            \[
                \mathfrak{R}^{(p)}(\ba{B}) \coloneqq \colim_{i \in I} \mathfrak{R}^{(p)}(\ba{B}_i).
            \]
        When $p=2$, this construction coincides with the one in \zcref{ssec:construction}: $\mathfrak{R}^{(2)}(\ba{B}) \cong \mathfrak{R}(\ba{B})$. For atomless Boolean algebras $\ba{B}$, a modification of the proof above shows that $\Idem{\mathfrak{R}^{(p)}(\ba{B})}$ is not spatial, where instead we need a $p$-fold product of ideals in \zcref{lem:orthogonalideals}. In the special case of $\ba{B} = \ba{F}_{\omega}$ and $k = \bbF_p$, the ring $\mathfrak{R}^{(p)}(\ba{F}_{\omega})$ identifies with the ring considered by Hebestreit and Scholze (\zcref{exa:HSring}) for arbitrary $p$.
    \end{Rem}
    \begin{Rem}
        The nil-Cantor rings $\R{\ba{B}}$ for any atomless Boolean algebra $\ba{B}$ also give new examples of rings which fail the telescope conjecture. Indeed, compactly generated localizing $\otimes$-ideals are in bijection with thick $\otimes$-ideals, and by Thomason's theorem, the latter are classified by Thomason subsets of $\Spec(\R{\ba{B}}) = \{ \mathfrak m_{\ba{B}} \}$, see \zcref{rem:localringalt}. In other words, there are exactly two compactly generated $\otimes$-ideals, namely $0$ and $\Der(\R{\ba{B}})$. On the other hand, $\Phi_{\R{\ba{B}}}(\mathfrak m_{\ba{B}})$ is a nonzero proper smashing ideal, since $0 < \mathfrak m_{\ba{B}} < \R{\ba{B}}$ and $\Phi_{\R{\ba{B}}}$ is injective. Hence this smashing ideal is not compactly generated. 
        
        In fact, in light of \zcref{exa:large}, \zcref{rem:localringalt}, \zcref{thm:frame-embedding}, and \zcref{lem:idempotentidealmachine2}, the commutative ring $\R{\ba{F}_{\kappa}}$ satisfies
            \[
                \lvert\mathrm{Thick}_{\otimes}(\Der(\R{\ba{F}_{\kappa}})^{\omega})\rvert = 2 \quad \text{and} \quad \lvert \Sm(\Der(\R{\ba{F}_{\kappa}})) \rvert \geq \kappa
            \]
        for any infinite cardinal $\kappa$. In this sense, the discrepancy measuring the failure of the telescope conjecture can be arbitrarily large.
    \end{Rem}
    \printbibliography
    \end{document}